\documentclass[11pt]{amsart}

\usepackage[T1]{fontenc}
\usepackage[utf8]{inputenc}
\usepackage{lmodern}
\usepackage{microtype}
\usepackage{amsmath,amssymb,mathtools}
\usepackage{enumitem}
\usepackage[hidelinks]{hyperref}
\usepackage[a4paper,margin=31mm]{geometry}

\newtheorem{theorem}{Theorem}[section]
\newtheorem{lemma}[theorem]{Lemma}
\newtheorem{proposition}[theorem]{Proposition}
\newtheorem{corollary}[theorem]{Corollary}
\theoremstyle{definition}
\newtheorem{definition}[theorem]{Definition}
\theoremstyle{remark}

\newtheorem{example}[theorem]{Example}

\newcommand{\Z}{\mathbb Z}

\newcommand{\Primes}{\mathcal P}

\title{Hopficity of profinite completions of abelian groups}

\author{Mattia Brescia}
\author{Ernesto Ingrosso}
\author{Marco Trombetti}

\address{Dipartimento di Matematica e Applicazioni ``Renato Caccioppoli'',
Universit\`a degli Studi di Napoli Federico II, Complesso Universitario Monte S. Angelo,
Via Cintia, Napoli, Italy}
\email{mattia.brescia@unina.it}
\email{ernesto.ingrosso2@unina.it}
\email{marco.trombetti@unina.it}


\begin{document}

\begin{abstract}
We determine exactly when the profinite completion of an arbitrary abelian
 group is topologically Hopfian.  For an abelian group $A$, we prove that
\[
  \widehat A \text{ is topologically Hopfian}
  \quad\Longleftrightarrow\quad
  A/pA \text{ is finite for every prime }p.
\] As a byproduct, we answer Problem 6.30 of the Kourovka Notebook in the
negative: for pairwise distinct odd primes $q_i$, the group
$\bigoplus_{i\geq1}\Z[1/q_i]$ is residually finite and Hopfian, whereas its
profinite completion is not topologically Hopfian.
\end{abstract}

\maketitle

\section{Introduction}

A surjective endomorphism of a finitely generated residually finite group is
injective, and a continuous surjective endomorphism of a topologically
finitely generated profinite group is likewise injective.  It is therefore
natural to ask how much of this rigidity remains when finite generation is
abandoned and one passes from a discrete group to its profinite completion. The main result of this paper gives a complete answer in the abelian case.  

\medskip

\noindent{\bf Main Theorem}\quad  {\it Let $A$ be an abelian group, and $\widehat A$ its profinite completion.  The following conditions are equivalent\textnormal:
\begin{enumerate}[label=\textup{(\roman*)}]
\item $\widehat A$ is topologically Hopfian\textnormal;
\item $A/pA$ is finite for every prime $p$\textnormal;
\item every Sylow pro-$p$ subgroup of $\widehat A$ is topologically finitely
      generated.
\end{enumerate}
}

\medskip

The theorem has a direct consequence for a problem of~Mel'nikov.  In
the sixth issue of the Kourovka Notebook, published in 1978, he asked whether
the profinite completion of every residually finite Hopfian group must be
Hopfian as a topological group; this is~Problem 6.30 in the current
edition~\cite{Kourovka}.  Consider the abelian group
\[
        G=\bigoplus_{i\geq1}\Z[1/q_i],
\]
where the $q_i$ are pairwise distinct odd primes. This is clearly Hopfian and residually finite. On the other hand,
$G/2G\cong\bigoplus_{i\geq1}C_2$ is infinite, so our Main Theorem 
yields that $\widehat G$ is not topologically Hopfian.

\smallskip

We use standard facts about profinite groups from
Ribes--Zalesskii~\cite{RibesZalesskii}, elementary abstract group theory from
Robinson~\cite{Robinson}, and the classical existence theorem for \hbox{$p$-basic} subgroups from Fuchs~\cite{Fuchs}. 

\section{Preliminaries}

An abstract group is \emph{Hopfian} if every surjective endomorphism is
injective.  A topological group is \emph{topologically Hopfian} if the same is
true for continuous endomorphisms. When referring to a profinite group to be Hopfian we usually omit the word ‘‘topologically’’ if it is clear from the context.  All homomorphisms between profinite groups
below are understood to be continuous.

Let \(A\) be an abelian group. Recall that its {\it profinite completion} is defined by
\[
\widehat{A}
   := \varprojlim_{\substack{N \leq A \\ [A:N] < \infty}} A/N.
\]
For each prime \(p\), we write
\[
\widehat{A}_p
   := \varprojlim_{\substack{N \leq A \\ [A:N] \text{ is a power of } p}} A/N
\]
for the {\it pro-\(p\) completion} of \(A\). Equivalently, \(\widehat{A}_p\) is the
maximal pro-\(p\) quotient of \(\widehat{A}\). Since \(\widehat{A}\) is abelian,
it decomposes canonically as
\[
\widehat{A} \cong \prod_{p} \widehat{A}_p,
\]
and \(\widehat{A}_p\) identifies with the Sylow pro-\(p\) subgroup of
\(\widehat{A}\). The following two results are easy and probably well-known.

\begin{lemma}\label{lem:tfg-hopfian}
Every topologically finitely generated profinite group is topologically
Hopfian.
\end{lemma}


\begin{lemma}\label{lem:primary}
Let $K$ be an abelian profinite group and let $K_p$ be its~Sy\-low pro-$p$
subgroup.  Then every endomorphism of $K$ preserves each $K_p$, and $K$ is Hopfian if and only if $K_p$ is Hopfian for every $p$.
\end{lemma}

\begin{definition}
Let $A$ be an abelian group and let $p$ be a prime.  A subgroup $B\leq A$ is
called \emph{$p$-basic} if:
\begin{enumerate}[label=\textup{(\roman*)}]
\item $B$ is a direct sum of infinite cyclic groups and finite cyclic
      $p$-groups;
\item $B$ is $p$-pure in $A$, that is,
      $B\cap p^mA=p^mB$ for every $m\geq1$;
\item $A/B$ is $p$-divisible, that is, $p(A/B)=A/B$.
\end{enumerate}
\end{definition}

We need the classical $p$-basic subgroup theorem of Kulikov.

\begin{theorem}\label{thm:basic-existence}
Every abelian group has a $p$-basic subgroup for every prime $p$.
\end{theorem}

A proof and further structure theory can be found in the chapter on purity
and basic subgroups in~\cite{Fuchs}. We also need the following basic fact, the proof of which we report.

\begin{lemma}\label{lem:reduction}
Let $B$ be a $p$-basic subgroup of an abelian group $A$.
\begin{enumerate}[label=\textup{(\roman*)}]
\item For every finite abelian $p$-group $F$, restriction induces a bijection
\[
        \operatorname{Hom}(A,F)
        \longrightarrow
        \operatorname{Hom}(B,F).
\]
\item The inclusion $B\hookrightarrow A$ induces an isomorphism
\(
        \widehat B_p\cong\widehat A_p.
\)
\item There is a natural isomorphism
\(
        B/pB\cong A/pA.
\)
\end{enumerate}
\end{lemma}

\begin{proof}
Let $F$ have exponent dividing $p^m$.  If a homomorphism $A\to F$ vanishes on
$B$, it factors through the $p$-divisible group $A/B$ and is therefore zero:
for $a+B=p^m(x+B)$ its value is $p^m$ times the value of $x+B$.
Thus restriction is injective.

Conversely, let $g:B\to F$.  Since $A/B$ is $p^m$-divisible, every $a\in A$
can be written
\[
        a=b+p^m x
        \qquad (b\in B,\ x\in A).
\]
Define $\widetilde g(a)=g(b)$.  If also $a=b'+p^m x'$, then
$b-b'\in B\cap p^mA=p^mB$, and hence $g(b-b')=0$.  Thus $\widetilde g$ is
well defined; it is plainly a homomorphism extending $g$.  This proves (i).

The induced continuous homomorphism
\[
\theta:\widehat B_p\longrightarrow\widehat A_p
\]
has dense image, since the image of \(B\) is dense in \(\widehat A_p\).
Moreover, \(\theta\) is injective: if \(0\neq x\in\widehat B_p\), then some
finite \(p\)-quotient of \(B\) does not annihilate \(x\), and by~\textup{(i)}
this quotient extends to \(A\), so \(\theta(x)\neq0\).

Finally, \(\widehat B_p\) is compact and \(\widehat A_p\) is Hausdorff.
Hence \(\theta(\widehat B_p)\) is compact, and therefore closed in
\(\widehat A_p\). Since this image is also dense, it follows that
\(
\theta(\widehat B_p)=\widehat A_p.
\)
Thus \(\theta\) is an isomorphism of pro-\(p\) groups.  This proves (ii).

Finally, $A/B$ is $p$-divisible, so $A=B+pA$.  Together with
$B\cap pA=pB$, the second isomorphism theorem gives
\[
        A/pA\cong B/(B\cap pA)=B/pB.
\] The statement is proved.
\end{proof}

\section{Proof of the main theorem}

Let $B$ be a $p$-basic subgroup of $A$.  By definition there are index sets
$I_0,I_1,I_2,\ldots$ such that
\begin{equation}\label{eq:basic-decomp}
 B\cong
 \Z^{(I_0)}\oplus
 \bigoplus_{n\geq1} C_{p^n}^{(I_n)}.
\end{equation}
Let
\[
        I=I_0\sqcup I_1\sqcup I_2\sqcup\cdots.
\]
Reduction modulo $p$ and Lemma~\ref{lem:reduction} give
\begin{equation}\label{eq:modp-basic}
        A/pA\cong B/pB\cong C_p^{(I)}.
\end{equation}
Thus $A/pA$ is finite if and only if $I$ is finite, equivalently if and only
if $B$ is finitely generated.

\begin{proposition}\label{prop:shift}
If the index set \(I\) in~\eqref{eq:basic-decomp} is infinite, then
\(\widehat B_p\) admits a continuous surjective non-injective endomorphism.
\end{proposition}

\begin{proof}
Write $B=\bigoplus_{i\in I} B_i$, where
\[
B_i\cong
\begin{cases}
\mathbb Z, & i\in I_0,\\
C_{p^n}, & i\in I_n.
\end{cases}
\]
Assign to each \(i\in I\) the height
\[
h(i)=
\begin{cases}
\infty, & i\in I_0,\\
n, & i\in I_n.
\end{cases}
\]

Since \(I\) is infinite, there exist distinct indices $i_1,i_2,\ldots$
such that $h(i_1)\leq h(i_2)\leq\cdots$. In fact, either one height occurs infinitely often, or infinitely many distinct
finite heights occur. For every \(k\geq 1\), the inequality $h(i_k)\leq h(i_{k+1})$ gives a surjective homomorphism
\[
\pi_k:B_{i_{k+1}}\twoheadrightarrow B_{i_k}.
\] Define an endomorphism
\[
\sigma:B\longrightarrow B
\]
by
\[
\sigma|_{B_{i_1}}=0,
\qquad
\sigma|_{B_{i_{k+1}}}=\pi_k
\quad (k\geq1),
\]
and let \(\sigma\) be the identity on every summand whose index does not belong
to \(\{i_1,i_2,\ldots\}\).

The map \(\sigma\) is surjective. Indeed, every unselected summand is fixed,
while each selected summand \(B_{i_k}\) is the image of \(B_{i_{k+1}}\).
It is not injective, since $B_{i_1}\leq \ker \sigma$.

Clearly, \(\sigma\) is continuous with respect to the
pro-\(p\) topology on \(B\).  Let
\[
\iota:B\longrightarrow\widehat B_p
\]
be the canonical map. Since \(\widehat B_p\) is a pro-\(p\) group, the
universal property of the pro-\(p\) completion applied to the continuous
homomorphism
\[
\iota\circ\sigma:B\longrightarrow\widehat B_p
\]
gives a unique continuous homomorphism
\[
\widehat{\sigma}:\widehat B_p\longrightarrow\widehat B_p
\]
such that
\(
\widehat{\sigma}\circ\iota=\iota\circ\sigma.
\)
Since \(\sigma(B)=B\), the image of \(\widehat{\sigma}\) contains the dense
subgroup \(B\) of~\(\widehat B_p\). The image of a compact group under a
continuous map is compact, hence closed, and therefore
\(
\widehat{\sigma}(\widehat B_p)=\widehat B_p.
\)
Thus \(\widehat{\sigma}\) is surjective.

Finally, \(B\) is residually a finite \(p\)-group, so the canonical map \(
B\longrightarrow\widehat B_p
\) is injective. Any non-zero element of \(B_{i_1}\) therefore remains non-zero
in \(\widehat B_p\) and belongs to \(\ker\widehat{\sigma}\). Consequently,
\(\widehat{\sigma}\) is not injective.
\end{proof}

\begin{theorem}\label{thm:local}
Let $A$ be an abelian group and let $p$ be a prime.  The following conditions are equivalent\textnormal:
\begin{enumerate}[label=\textup{(\roman*)}]
\item $\widehat A_p$ is topologically Hopfian\textnormal;
\item $A/pA$ is finite\textnormal;
\item $\widehat A_p$ is topologically finitely generated\textnormal;
\item every $p$-basic subgroup of $A$ is finitely generated.
\end{enumerate}
\end{theorem}

\begin{proof}
Choose a $p$-basic subgroup $B$ and write it as in
\eqref{eq:basic-decomp}.  By~\eqref{eq:modp-basic}, conditions (ii) and~(iv)
are equivalent.

If $B$ is finitely generated, then
\(
        B\cong\Z^{r_p}\oplus F_p
\)
with $r_p<\infty$ and $F_p$ a finite abelian $p$-group.  By
Lemma~\ref{lem:reduction},
\(\widehat A_p\cong\widehat B_p\). Thus (iv) implies (iii), and Lemma~\ref{lem:tfg-hopfian} gives
(iii)$\Rightarrow$(i).

If $B$ is not finitely generated, then the index set $I$ is infinite and
Proposition~\ref{prop:shift}, together with
$\widehat A_p\cong\widehat B_p$, shows that $\widehat A_p$ is not Hopfian.
Hence (i) implies (iv).  All conditions are therefore equivalent.
\end{proof}

\begin{corollary}\label{cor:p-primary}
Let $A$ be an abelian $p$-group.  Then the following are equivalent:
\[
 \widehat A\text{ is Hopfian},\qquad
 A/pA\text{ is finite},\qquad
 \widehat A\text{ is finite}.
\]
Equivalently, $A=D\oplus F$, where $D$ is divisible and $F$ is a finite
abelian $p$-group.
\end{corollary}

\begin{proof}
A $p$-basic subgroup $B$ of a $p$-group has no infinite cyclic summands.  If
$A/pA$ is finite, then Theorem~\ref{thm:local} shows that $B$ is finitely
generated and hence finite.  Since $B$ is pure and bounded, it is a direct
summand of $A$; write $A=B\oplus D$.  The quotient $A/B\cong D$ is divisible,
so $A=D\oplus B$ with $B$ finite.  Moreover,
$\widehat A\cong\widehat B=B$ by Lemma~\ref{lem:reduction}.  The converse is
immediate, since a finite completion is Hopfian and
Theorem~\ref{thm:local} then gives the finiteness of $A/pA$.
\end{proof}


\bigskip

\noindent{\it Proof of the Main Theorem}\quad By Lemma~\ref{lem:primary}, $\widehat A$ is Hopfian if and only if each
$\widehat A_p$ is Hopfian.  Apply Theorem~\ref{thm:local} prime by prime.\qed

\bigskip

\begin{corollary}\label{cor:localizations}
Let $S_i$ be sets of primes and put
\(
        A=\bigoplus_{i\in I}\Z[S_i^{-1}].
\)
For a prime $p$, set
\(
        I_p=\{i\in I:p\notin S_i\}.
\)
Then
\[
        \widehat A\text{ is Hopfian}
        \quad\Longleftrightarrow\quad
        I_p\text{ is finite for every prime }p.
\]
\end{corollary}

\begin{proof}
Reduction modulo $p$ gives
\[
 \Z[S_i^{-1}]/p\Z[S_i^{-1}]
 \cong
 \begin{cases}
 0,&p\in S_i,\\
 C_p,&p\notin S_i.
 \end{cases}
\]
Thus $A/pA\cong\bigoplus_{i\in I_p}C_p$, and the conclusion follows from the Main Theorem.
\end{proof}

\begin{example}[Hopfian but not finitely generated]\label{ex:positive}
For each prime $p$, choose a finite integer $n_p\geq0$ and put
\[
        A=\bigoplus_{p\in\Primes}\Z_{(p)}^{\,n_p},
\]
where $\Z_{(p)}$ consists of rational numbers whose denominators are prime to
$p$.  Then
\[
        A/qA\cong C_q^{\,n_q}
        \quad\text{and}\quad
        \widehat A\cong\prod_{p\in\Primes}\Z_p^{\,n_p}.
\]
Hence $\widehat A$ is Hopfian.  If the $n_p$ are unbounded, then $\widehat A$
is not topologically finitely generated.
\end{example}

\section*{Acknowledgements}

The authors are members of the non-profit association ``AGTA -- Advances in
Group Theory and Applications'' (www.advgrouptheory.com) and are supported by GNSAGA (INdAM).

\end{document}